\documentclass[12pt]{article}
\usepackage[english]{babel}

\usepackage{amsfonts,amsmath,amsxtra,amsthm,amssymb,latexsym}
\usepackage{amscd,upref}
\usepackage{color}

\newcommand{\ack}{\section*{Acknowledgments}}
\theoremstyle{plain}
\newtheorem{theorem}{Theorem}[section]
\newtheorem{corollary}[theorem]{Corollary}
\newtheorem{lemma}[theorem]{Lemma}

\newtheorem{definition}{Definition}[section]
\newtheorem{remark}{Remark}[section]
\newtheorem{example}{Example}[section]
\newtheorem{proposition}[theorem]{Proposition}

\numberwithin{equation}{section}

\DeclareMathOperator{\sgn}{sgn} \DeclareMathOperator{\dom}{dom}
 
\DeclareMathOperator{\Span}{span}

\DeclareMathOperator{\loc}{loc}
\DeclareMathOperator{\tr}{tr}\DeclareMathOperator*{\esssup}{ess\,sup}

\newcommand{\ti}{\tilde}
\newcommand\R{{\mathbb{R}}}
\newcommand\CC{{\mathbb{C}}}

\newcommand\Z{{\mathbb{Z}}}

\newcommand\gH{{\mathfrak{H}}}
\newcommand\gt{{\mathfrak{t}}}

\newcommand\cM{{\mathcal{M}}}

\newcommand\cK{{\mathcal{K}}}
\newcommand\cG{{\mathcal{G}}}

\newcommand\cW{{\mathcal{W}}}

\newcommand\rD{{\rm{d}}}
\newcommand\I{{\rm{i}}}

\title{The similarity problem for indefinite Sturm--Liouville  operators with periodic coefficients}

\author{Aleksey Kostenko}

\date{}

\begin{document}

\maketitle

\begin{abstract}
We investigate the problem of similarity to a self-adjoint operator for $J$-positive Sturm--Liouville operators $L=\frac{1}{\omega}\left(-\frac{d^2}{dx^2}+q\right)$ with $2\pi$-periodic coefficients $q$ and $\omega$. It is shown that if $0$ is a critical point of the operator $L$, then it is a singular critical point. This gives us a new class of $J$-positive differential operators with the singular critical point $0$. Also, we extend the Beals and Parfenov regularity conditions for the critical point $\infty$ to the case of operators with periodic coefficients.
\end{abstract}

\quad


\quad


\quad

\section{ Introduction}\label{intro}

Consider the Sturm--Liouville spectral problem
\begin{equation}\label{I_01}
-y''(x)+q(x)y(x)=z\ \omega(x)y(x),\qquad x\in \R,
\end{equation}
with real $2\pi$-periodic coefficients $q(x)=q(x+2\pi)$, $\omega(x)=\omega(x+2\pi)$, and such that $q=\overline{q}\in L^1(0,2\pi)$, $\omega=\overline{\omega}\in L^1(0,2\pi)$, and $|\omega|>0$ a.e. on $\R$. If the weight function $\omega$ does not change sign on $\R$, then \eqref{I_01} leads to a self-adjoint operator in the Hilbert space $L^2(\R,|\omega|)$,
\begin{equation}\label{I_02}
L:=\frac{1}{\omega}\left(-\frac{\rD^2}{\rD x^2}+q\right),\quad \dom(L):=\{f\in L^2(\R,\omega): \ f, f'\in AC_{\loc}(\R),\ Lf\in L^2(\R,|\omega|) \}.
\end{equation}
Spectral properties of $L$ are widely studied in this case (see \cite{RS_IV, KhoRof, Tit58, Weid87} and references therein). In particular, the spectrum $\sigma(L)$ consists of a countable set of closed intervals on the real axis, which may degenerate into finitely many closed intervals and a half-line. The spectral expansion was constructed by Gelfand \cite{Gel50} and Titchmarsh \cite{Tit58}.

Much less is known if the weight function changes sign \cite{LanDah86, Zet05}.
In this case, the operator $L$ associated with \eqref{I_01} is no longer self-adjoint or even symmetric in  $L^2(\R,|\omega|)$. Moreover, its spectrum is not necessarily real. However, the differential operator \eqref{I_02} is symmetric with respect to the indefinite inner product, so it is natural to consider it in the Krein space $L^2(\R,\omega)$ (for definitions see \cite{Lan82}).

Sturm--Liouville differential operators and higher order ordinary differential operators with indefinite
weight functions have attracted a lot of attention in the recent past \cite{BTrunk07, BC04, BF11, CurLan, CN95, LanDah86, Fle96, Fle_07, Fle_08, Kar_09,KarKos08, KarKos09, KAM_09, KM06, Par03, Pyat89, Pyat05, Vol96, Zet05}.
The main problem we are concerned with is
\emph{the similarity} of the operator $L$
to a self-adjoint operator. The similarity of the operator $L$ to a
self-adjoint one is essential for the theory of forward-backward
parabolic equations arising in certain physical models (see \cite{Beals85, Kar_07} and references therein).

Recall that closed operators $T_1$
and $T_2$ in a Hilbert space $\mathfrak{H}$ are called \emph{similar} if
there exists a bounded operator $S$ with bounded inverse
$S^{-1}$ in $\mathfrak{H}$ such that $S\dom(T_1) = \dom(T_2)$ and $T_2 = S
T_1 S^{-1}$.

In the present paper, we restrict ourselves to the case when the operator $L$ is $J$-positive, i.e., $J:f(x)\to (\sgn \omega(x)) f(x)$ and throughout the paper we assume that the operator
\begin{equation}\label{I_03}
A:=JL=\frac{1}{|\omega|}\left(-\frac{\rD^2}{\rD x^2}+q\right),\qquad \dom(A):=\dom(L),
\end{equation}
is positive in $L^2(\R,|\omega|)$\footnote{We say that a closed operator $T$ in a Hilbert space $\mathfrak{H}$ is positive, $T>0$, (nonnegative, $T\ge 0$) if $(Tf,f)_{\mathfrak{H}}>0$ \ ($(Tf,f)_{\mathfrak{H}}\ge 0$) for every $f\in\dom(T)\setminus\{0\}$}. Note that (see \cite{Weid87}) the condition $\int_0^{2\pi} q(x)dx>0$ is necessary for the operator $A$ to be positive if $q\neq 0$. Let us also stress that $\ker A=\ker L=\{0\}$ even if $0\in \sigma(L)$ since the spectrum of $L$ is purely continuous (cf. Theorem \ref{th_III.1}).

For regular $J$-nonnegative differential operators (or, more generally, for operators with discrete spectra) the similarity problem is equivalent to the Riesz basis property of eigenfunctions. The Riesz basis property  has been studied in \cite{Beals85,  BC04, CurLan, Fle96, Fle_07, Fle_08, Par03, Pyat89, Pyat05, Vol96} (see also the paper by P. Binding and A. Fleige \cite{BF11} in this volume  for a review) and quite general sufficient conditions have been found, the so-called Beals' conditions. It turns out that the Riesz basis property essentially depends on local behavior of $\omega$ in a neighborhood of turning points. Moreover, in \cite{Par03} (see also \cite{BF11}) it was shown that these conditions are close to be necessary
(for further details see Section \ref{sec_II} below). As was pointed out in \cite{CurLan} the analysis extends  to the case when the $J$-nonnegative operator $L$ has a bounded inverse, i.e., $0\in\rho(L)$. In this case, the similarity problem is equivalent to the problem of regularity of the critical point $\infty$ of the $J$-positive operator $L$.

If $0\in\sigma(L)$, then $0$ may be a critical point of $L$ and the situation is more complicated in this case. Several
abstract similarity criteria
can be found in \cite{Cur_85,MMMCr,NabCr,Ves72}, however, it is not easy to apply them to operators
of the form \eqref{I_02}.
The first result of this type
was obtained by B.~\'Curgus and B.~Najman \cite{CN95}. Using the \'Curgus regularity criterion \cite[Theorem 3.2]{Cur_85}, they proved that  the operator $(\sgn x) \frac{d^2}{d x^2}$ is similar to a self-adjoint operator in $L^2(\R)$. This result has substantially been extended in \cite{KAM_09, KM06} (see also references therein). Namely, using the resolvent criterion of
similarity (see \cite{MMMCr,NabCr}), the regularity of the critical point $0$ was established for several classes of indefinite Sturm--Liouville operators. Also, in \cite{KarKos08, KarKos09} (see also \cite[Section 5]{KAM_09} and \cite[Section 5]{Kar_09}) it was shown that $0$ can be a singular critical point and examples of such indefinite Sturm--Liouville operators have been given.

Let us emphasize that all the above mentioned conditions have been obtained under the additional assumption that $\omega$ has only a finite number of turning points, i.e., points where $\omega$ changes sign (there is an example given by Pyatkov \cite{Pyat89} with infinitely many turning points and no Riesz basis of eigenfunctions, see also Remark \ref{rem_inf_c_p} below). So, the operator $L$ with periodic coefficients and an indefinite weight is the simplest and, in some sense, natural model of operators with an infinite number of turning points.


This paper is also motivated by the following problem posed by B. \'Curgus during his talk at the $6^{th}$ Workshop on Operator Theory in Krein Spaces. He suggested \cite{Cur_06} that the operator defined in $L^2(\R)$ by the differential expression
\begin{equation}\label{I_04}
-(\sgn \sin x)\frac{\rD^2}{\rD x^2}
\end{equation}
\emph{is a good candidate to be similar to a self-adjoint one}. In the present paper we give the answer to this problem (see Example \ref{ex_I}). Namely, we show that {\em this operator is a bad candidate}, that is, {\em the operator \eqref{I_04} is not similar to a self-adjoint one}. The reason is the following. We show that if $0$ is a critical point for the $J$-positive Sturm--Liouville operator $L$ with periodic coefficients, then it is a singular critical point (see Theorem \ref{th_3.3}). In particular, in the case $q\equiv 0$, $0$ is a critical point for $L$ precisely when $\int_0^{2\pi}\omega(x)dx=0$ (see \cite[Theorem 5.1]{LanDah86} and also Corollary \ref{cor:3.7}). The latter immediately yields that in the case $q\equiv 0$, $0$ is a singular critical point for $L$ if and only if $\int_0^{2\pi}\omega(x)dx=0$ (Corollary \ref{c5.3}). Thus the conjecture of B. \'Curgus unexpectedly leads to a new class of differential operators with the singular critical point $0$.

Also we extend the Beals and Parfenov regularity conditions for the critical point $\infty$ to the case of operators with periodic coefficients under the additional assumption that $\omega$ has only a finite number of critical points on $[0,2\pi)$ (Theorem \ref{th_inftycond}). Combining these results we obtain necessary and sufficient conditions for the operator $L$ to be similar to a self-adjoint one (Theorem \ref{th_3.4}). Our approach is based on the Gelfand decomposition of operators with periodic coefficients \cite{Gel50, RS_IV} (see also Subsection \ref{ss_III.1}). This representation plays a key role in the study of Hill operators $H=-d^2/dx^2+q(x)$ with complex potentials $q\neq \overline{q}$. In \cite{McG_65}, McGarvey noticed that for the spectral projections of $H$ to be bounded at $\infty$ it is necessary, roughly speaking, that the associated family $H(t)$ (cf. \eqref{3.36}) has a uniform Riesz basis property.
Veliev (see \cite{Vel_83}) was the first to note that the finite point $\lambda_0\in\CC$ is the spectral singularity of the Hill operator $H$ if and only if the root subspace of $H(t)$ corresponding to $\lambda_0$ contains a root function which is not an actual eigenfunction of $H(t)$. In the recent paper \cite{GT_09}, a criterion for the similarity to a normal operator was obtained for Hill operators with complex locally square integrable potentials.

To conclude, we briefly describe the content of the paper. Section \ref{sec_II} contains necessary information on conditions for the Riesz basis property for indefinite Sturm--Liouville problems on a finite interval. In Subsection \ref{ss_III.1}, we collect basic facts on the Floquet theory of second order differential operators with periodic coefficients. Subsection \ref{ss_III.2} contains the main results of the paper, necessary and sufficient regularity conditions of critical points $0$ and $\infty$ for the operator $L$, as well as necessary and sufficient conditions for the operator $L$ to be similar to a self-adjoint one. In the final Subsection \ref{ss_exmpl} we illustrate the results by considering two simple examples. In particular, we discuss the conjecture of B. \'Curgus.

\textbf{Notation.} Throughout this paper we use the symbol
$\prime$ to denote $x$-derivatives and $\bullet$ to denote
$z$-derivatives. Moreover, we  denote by $\sigma(\cdot)$
 and $\rho(\cdot)$ the spectrum and the
resolvent set of a densely defined, closed, linear operator in a Hilbert
space. 

\section{Preliminaries} \label{sec_II}
\subsection{Sturm--Liouville operators on a finite interval}\label{ss_II.1}

We begin with the Sturm--Liouville differential expression on a finite interval
\begin{equation}\label{II_1_01}
\ell[y]:=\frac{1}{\omega(x)}\left(-y''+q(x)y\right),\qquad x\in [0,2\pi].
\end{equation}
It is assumed that $q=\overline{q}\in L^1[0,2\pi]$, $\omega=\overline{\omega}\in L^1[0,2\pi]$, and $|\omega|>0$ a.e. on $[0,2\pi]$.
Let us associate with \eqref{II_1_01}  the operators $L_D$ and $L(t)$ defined in $L^2([0,2\pi];|\omega|)$ on the domains 
\begin{equation}\label{II_1_02}
 \dom(L_D)=\{f\in \dom(L_{\max}) :\ f(0)=f(2\pi)=0\},
\end{equation}
and, for $t\in [0,2\pi)$,
\begin{equation}\label{II_1_03}
 \dom\big(L(t)\big)=\{f\in \dom(L_{\max}) :\ f(0)=e^{it}f(2\pi),\ f'(0)=e^{it}f'(2\pi)\},
\end{equation}
respectively. Here
\[
\dom(L_{\max})=\{f\in L^2([0,2\pi];|\omega|):\ f, f'\in AC[0,2\pi],\ \ell[f]\in L^2([0,2\pi];|\omega|)\}.
\]
If the weight function $\omega$ is positive on $[0,2\pi]$, then $L_D$ and $L(t)$ are self-adjoint operators  in the Hilbert space $L^2([0,2\pi]; |\omega|)$. Moreover, the operators $L_D$ and $L(t)$ are lower semibounded  and their spectra are discrete (cf. \cite{Weid87, Zet05}). The form domains $\gt_D$ and $\gt(t)$ of $L_D$ and $L(t)$, respectively, are given by
\begin{equation}\label{II_1_02B}
 \dom(\gt_D)=\Big\{f\in L^2([0,2\pi]; |\omega|): \ f\in AC[0,2\pi],\ \int_{[0,2\pi]}|f'(x)|^2dx<\infty,\ f(0)=f(2\pi)=0\Big\},
\end{equation}
and, for $t\in [0,2\pi)$,
\begin{equation}\label{II_1_03B}
 \dom\big(\gt(t)\big)=\Big\{f\in L^2([0,2\pi]; |\omega|): \ f\in AC[0,2\pi],\ \int_{[0,2\pi]}|f'(x)|^2dx<\infty,\ f(0)=e^{it}f(2\pi)\Big\}.
\end{equation}
Note that the above description of the form domains was obtained by M.G. Krein \cite[\S 6]{kre46} in the case $\omega\equiv 1$. However, the methods used there extend to the case of arbitrary positive weights $\omega\in L^1[0,2\pi]$.

If $\omega$ is indefinite, then $L_D$ and $L(t)$ are self-adjoint operators in the Krein space $\mathcal{K}=L^2([0,2\pi];\omega)$ with the inner product
\[
[f,g]_{\mathcal{K}}:=\int_0^{2\pi}f(x)\overline{g(x)}\omega(x)dx.
\]
Let us denote by $\ti{J}$ the operator of multiplication by  $\sgn \omega(x)$ in $L^2([0,2\pi]; |\omega|)$. Clearly, $\ti{J}=\ti{J}^{-1}=\ti{J}^*$. In the following we assume that $L_D$ and $L(t)$ are $\ti{J}$-nonnegative operators in $L^2([0,2\pi]; |\omega|)$, i.e.,
\[
[L_Df,f]_{\mathcal{K}}\geq 0\quad \text{for}\quad f\in \dom(L_D),\quad \text{and}\quad  [L(t)f,f]_{\mathcal{K}}\geq 0 \quad \text{for} \quad f\in \dom \big(L(t)\big),\ \ t\in[0,2\pi).
\]
Note that the spectra of the operators $L_D$ and $L(t)$ are discrete since so are the spectra of $A_D=\ti{J}L_D$ and $A(t)=\ti{J}L(t)$, respectively. Thus the resolvent sets $\rho(L_D)$ and $\rho\bigl(L(t)\bigr)$ are nonempty and hence
the operators $L_D$ and $L(t)$ have real spectra (cf. \cite[Proposition 2.2]{CurLan}).
Note also that only $0$ may be a nonsemisimple eigenvalue of $L_D$ and $L(t)$. Moreover, the algebraic multiplicity of $0$ is not greater than $2$ (see, e.g., \cite{Fle96, Lan82}).

Further (see \cite{Lan82}), a definitizable operator $\mathcal{A}$ in a Krein space $\cK$ has a spectral function $E_\mathcal{A}$. This function has the properties similar to the properties of a spectral function of a self-adjoint operator in a Hilbert space. The main difference is the occurrence of
\emph{critical points}. Significantly different behavior of the
spectral function $E_{\mathcal{A}} (\cdot)$ occurs at a \emph{singular
critical point} in any neighborhood of which
$E_{\mathcal{A}} (\cdot)$ is unbounded. A critical point is  \emph{regular} if it is not singular. It should be stressed that only $0$ and $\infty$ may be critical points of a definitizable $J$-nonnegative operator $\mathcal{A}$. Furthermore, $\mathcal{A}$ is similar to a self-adjoint operator if and only if $\ker \mathcal{A}=\ker \mathcal{A}^2$ and all its critical points are not singular. 

We emphasize that $\infty$ is a critical point for the operators $L_D$ and $L(t)$. In the last three decades, the question of the regularity of $\infty$ has been intensively studied \cite{Beals85, BC04, BF11, CurLan, Fle96, Fle_07, Fle_08, Par03, Pyat89, Pyat05, Vol96} (see also references therein). It turned out that the answer significantly depends on local behavior of the weight  $\omega$. Firstly, recall the following notion (cf. \cite[Remark 3.3]{CurLan}).
\begin{definition}[\cite{CurLan}]\label{def_II.1}
Point $x_0\in\R$ is called \emph{turning} if $(x-x_0)\omega(x)$ is of one sign for almost all $x$ in some neighborhood of $x_0$.
Assume that $\omega$ is absolutely continuous on $(x_0, x_0+\varepsilon]$ and $[x_0-\varepsilon,x_0)$ for some $\varepsilon>0$. Turning point $x_0\in \R$ is called \emph{simple} if there exists  $s_1>0$, $s_1\neq 1$, such that the function $\bigl(\omega(x)/\omega(s_1x)\bigr)'$ is bounded in a neighborhood of $x_0$ and
\begin{equation} \label{e BC Simple}
\lim_{x\downarrow x_0}\frac{\omega(x)}{\omega(s_1x)}\neq s_1.
\end{equation}
\end{definition}
\begin{remark}\label{rem_II.1}
Definition \ref{def_II.1} implies that a turning point $x_0$ is simple if there exist $\beta_\pm>-1$ and positive functions $ p_+ \in C^1 [x_0,x_0+\delta]$, $p_-\in C^1 [x_0-\delta,x_0]$ such that
\begin{equation} \label{e BC Simple_sec}
\omega(x)= \sgn (x-x_0) p_\pm (x)|x-x_0|^{\beta_\pm} , \quad \pm (x-x_0) \in
(0,\delta).
\end{equation}
\end{remark}
Combining \cite[Theorem 3.6]{CurLan}, \cite[Theorem 6]{Par03}, \cite[Theorem 4.2]{Pyat05}, and \cite[Theorem 3.1]{Fle_07}, we arrive at the following necessary and sufficient conditions for the critical point $\infty$ to be regular.

\begin{theorem}\label{th_II.1}
Assume that $\omega$ has a finite number of turning points on $[0,2\pi)$. Assume also that $\omega$ has the same sign in neighborhoods of $x=0$ and $x=2\pi$.
\item $(i)$ $\infty$ is a regular critical point for the operators $L_D$ and  $L(t)$ if all turning points of $\omega$ are simple.
\item $(ii)$ Assume, in addition, that $\omega$ is odd\footnote{the function $\omega$ is called odd in the neighborhood of $x_0$ if there exists $\varepsilon_0>0$ such that $\omega(x_0+\varepsilon)=-\omega(x_0-\varepsilon)$ for all $\varepsilon\in(0,\varepsilon_0)$} and continuously differentiable in a punctured neighborhood of each turning point $x=x_0$, and the following limit exists
    \begin{equation}\label{eq:fleige}
    \lim_{x\downarrow x_0}\left(\frac{\omega(x)}{\mu\omega(\mu x)}\right)'=\lim_{x\downarrow x_0}\frac{\omega'(x)\omega(\mu x)-\mu\omega(x)\omega'(\mu x)}{\mu\omega(\mu x)^2}.
    \end{equation}
     Then $\infty$ is a regular critical point of $L_D$ and $L(t)$ if and only if all turning points are simple.
\end{theorem}
\begin{proof}
In the case $q\equiv 0$, Theorem \ref{th_II.1} follows by combining \cite[Theorem 3.6]{CurLan}, \cite[Theorem 6]{Par03}, \cite[Theorem 4.2]{Pyat05}, and \cite[Theorem 3.1]{Fle_07}. Namely, since $\omega$ has a finite number of turning points on $[0,2\pi)$, by Theorem 4.2 from \cite{Pyat05} the Riesz basis property depends on a local behavior of the weight function $\omega$ in a neighborhood of its turning points. In particular, by \cite[Theorem 3.6]{CurLan} ,$\infty$ is a regular critical point for the operators $L_D$ and  $L(t)$ if all turning points are simple in the sense of Definition \ref{def_II.1}. If $\omega(x)$ is odd in a punctured neighborhood of each turning point, then \cite[Theorem 6]{Par03} and \cite[Theorem 3.1]{Fle_07} prove $(ii)$ in the case $q\equiv 0$.

To prove Theorem \ref{th_II.1} in the general case, observe that the form domain of the operator $\ti{J}L(t)$ given by \eqref{II_1_03B} does not depend on $q\in L^1[0,2\pi]$. Hence Theorem 3.5 from \cite{Cur_85} completes the proof.
%
\end{proof}

\begin{remark}\label{rem_inf_c_p}
Let us note that the regularity of the critical point $\infty$ has not been studied yet for the operators $L(t)$ and $L_D$ in the case when the weight function $\omega$ has infinitely many turning points on $[0,2\pi)$. It is only known that in this case the situation is much more complicated. In \cite{Pyat89}, S.G. Pyatkov proved the following fact: if $\omega(x)=\sgn \sin\left(\frac{2\pi}{ x}\right)$ on $(0,2\pi)$, then \emph{$\infty$ is the singular critical point for $L_D$}. Let us stress that in this example all turning points are simple in the sense of Definition \ref{def_II.1}.
\end{remark}

\section{J-positive Sturm--Liouville operators with periodic coefficients}\label{sec_III}

\subsection{Floquet Theory}\label{ss_III.1}

In this subsection we briefly recall some standard results on second order differential operators with real periodic coefficients. This enables us to prove some basic spectral properties of indefinite Sturm--Liouville operators with periodic coefficients.

Let $c(\cdot,z)$ and $s(\cdot,z)$ be the fundamental system of
 solutions  of equation \eqref{I_01} satisfying
\begin{equation}
c(0,z)=s'(0,z)=1, \quad
c'(0,z)=s(0,z)=0, \quad z\in\CC.  \label{3.3}
\end{equation}
For each $x\in\R$, $c(x,z)$ and $s(x,z)$ are entire with respect
to $z$. The monodromy matrix
$\cM(z)$ is then given by
\begin{equation}
\cM(z)=\begin{pmatrix} c(2\pi,z) & s(2\pi,z) \\
c'(2\pi,z) & s'(2\pi,z) \end{pmatrix}, \quad
z\in\CC.  \label{3.4}
\end{equation}
Its eigenvalues $\rho_\pm(z)$, the Floquet multipliers, satisfy
$\rho_+(z)\rho_-(z)=1$
since $ \det(\cM(z))=1$.
The Floquet discriminant $\Delta(\cdot)$ is then defined by
\begin{equation}
\Delta(z)=\frac{1}{2}\tr(\cM(z))=\frac{c(2\pi,z)+s'(2\pi,z)}{2}, \quad z\in\CC,  \label{3.6}
\end{equation}
and one obtains
\begin{equation}
\rho_\pm (z)=\Delta(z)\pm i\sqrt{1-\Delta(z)^2}.  \label{3.7}
\end{equation}
with an appropriate choice of the square root branches. We also note that
$|\rho_\pm(z)|=1$  if and only if $\Delta(z)\in [-1,1]$.
Moreover,
\begin{equation}
\sigma(L_D)=\{z\in\CC:\ s(2\pi,z)=0\},\qquad \sigma(L(t))=\{z\in\CC:\ \Delta(z)=\cos t\}.\label{3.7D}
\end{equation}
The
Floquet solutions $\psi_\pm(\cdot,z)$ of $\ell y=z y$
normalized by $\psi_\pm(0,z)=1$ are then given by
\begin{align}\label{3.20A}
\psi_\pm(x,z)=c(x,z)+\frac{\rho_\pm(z)-c(2\pi,z)}{s(2\pi,z)}s(x,z)  =c(x,z)+m_\pm(z)s(x,z),\\
\quad m_\pm(z)=\frac {-[c(2\pi,z)-s'(2\pi,z)]/2\pm i\sqrt{1-\Delta(z)^2}}{s(2\pi,z)}, \quad z\in\CC\setminus\sigma(L_D). \label{3.20}
\end{align}
One then verifies (for $z\notin \sigma(L_D)$, $x\in\R$),
\begin{gather}
\psi_{\pm}(x+2\pi,z)=\rho_\pm(z) \psi_\pm(x,z)=e^{\pm it}\psi_{\pm}(x,z) \, \text{ with
$\Delta(z)=\cos(t)$,}   \label{l24} \\
 W(\psi_+(\cdot,z),\psi_-(\cdot,z))= m_-(z)- m_+(z)= -
\frac{2i\sqrt{1-\Delta (z)^2}}{s(2\pi,z)}.   \label{3.21} 
\end{gather}
where $W(f,g)=fg'-f'g$ is the Wronskian. Further, denote $\mathcal{D}:=\bigcup_{t\in[0,2\pi)} \sigma(L(t))\subseteq\R$ and choose the square root branches such that $|\rho_-(z)|>1$ and $|\rho_+(z)|<1$ on $\CC\setminus \mathcal{D}$. Then
\[
\psi_\pm(\cdot,z)\in L^2(\R_\pm,|\omega|),\qquad z\notin \mathcal{D}\cup \sigma(L_D).
\]

Next, we establish the connection between the original
Sturm--Liouville operator $L$ in $L^2(\R,|\omega|)$ and the family of operators
$\{L(t)\}_{t\in[0,2\pi)}$ in $L^2([0,2\pi];|\omega|)$ using the notion of a direct
integral and the Gelfand transform \cite{Gel50} (see also \cite[Chapter XIII.16]{RS_IV}). To this end we consider
the direct integral of Hilbert spaces with constant fibers $L^2([0,2\pi];|\omega|)$
(for details see \cite{BirSol_87, RS_IV})
\begin{equation}
\gH=\frac{1}{2\pi}\int^{\oplus}_{[0,2\pi]}  L^2([0,2\pi];|\omega|)\ dt.   \label{3.31}
\end{equation}
Elements $F\in\gH$ are represented
by $F=\big\{F(\cdot,t)\in L^2([0,2\pi];|\omega|)\big\}_{t\in[0,2\pi]}$ and
\begin{equation}
 \|F\|^2_{\gH} = \frac{1}{2\pi}
\int_{[0,2\pi]} dt\, \|F(\cdot,t)\|^2_{L^2([0,2\pi];|\omega|)}
= \frac{1}{2\pi}
\int_{[0,2\pi]} \int_{[0,2\pi]} |F(x,t)|^2|\omega(x)|dx\,  dt  \label{3.32}
\end{equation}
with scalar product in $\gH$ defined by
\begin{equation}\label{3.33}
(F,G)_{\gH} = \frac{1}{2\pi}
\int_{[0,2\pi]}  (F(\cdot,t),G(\cdot,t))_{L^2([0,2\pi];|\omega|)} dt, \quad F, G \in \gH. 
\end{equation}
The Gelfand transform $\cG\colon  L^2(\R;|\omega|) \to \gH$ is then defined by (cf. \cite{Gel50, RS_IV})
\begin{equation}
(\cG f)(x,t)=F(x,t)=s-\lim_{N \uparrow \infty}
\sum_{n=-N}^N  f(x+ 2\pi n) e^{-int},
     \label{3.34}
\end{equation}
where $s-\lim$ denotes the limit in $\gH$.
By inspection (see, e.g., \cite{Gel50} and \cite[proof of the Lemma on p.289]{RS_IV}), $\cG$ is a unitary operator and the inverse transform ${\cG}^{-1}: \gH \to L^2(\R, |\omega|)$ is given by
\begin{equation}
 ({\cG}^{-1}F)(x+ 2\pi n) = \frac{1}{2\pi}\int_{[0,2\pi]}
 F(x,t) e^{int}dt\, , \quad n\in\Z.   \label{3.35}
\end{equation}
Moreover, one then infers that
\begin{equation}
\cG L {\cG}^{-1} = \mathcal{L}:=\frac{1}{2\pi}\int^{\oplus}_{[0,2\pi]}
L(t) dt\, .
 \label{3.36}
\end{equation}
Namely, observe that it suffices to establish \eqref{3.36} in the case of definite weights, $\omega(x)>0$ for a.e. $x\in \R$. Indeed, the latter clearly follows from the relations $\cG J\cG =\frac{1}{2\pi}\int^\oplus_{[0,2\pi]}\ti{J} dt$ and $\ti{J}=\ti{J}^{*}=\ti{J}^{-1}$, where $\ti{J}$ is the operator of multiplication by $\sgn \omega(x)$ in $L^2([0,2\pi];|\omega|)$. So, assume that $\omega$ does not change sign on $\R$. Then $L$ is self-adjoint in $L^2(\R,|\omega|)$. Further, for a function $f\in \dom(L_{\min})$  the sum $F(x,t)=(\cG f)(x,t)$ given by \eqref{3.34} is finite and hence convergent. Here $\dom(L_{\min})$ is the minimal domain of the differential expression $\ell[\cdot]$,
\[
\dom(L_{\min})=\{f\in C^\infty_{\mathrm{comp}}(\R): \  \ell[f]\in L^2(\R;|\omega|)\}.
\]
Moreover, $F\in C^\infty$ and $\ell[F]=\cG(\ell[f])$ since the sum in \eqref{3.36} is finite. The latter also implies $\ell[F]\in \gH$ and $\ell[F(\cdot,t)]\in L^2([0,2\pi];|\omega|)$ for every $t\in [0,2\pi)$.  Finally, it is straightforward to check that $F(2\pi, t)=\mathrm{e}^{\I t}F(0,t)$ and
$F'_x(2\pi, t)=\mathrm{e}^{\I t}F'_x(0,t)$. Therefore, we get $F(\cdot,t)\in \dom \big(L(t)\big)$ for every $t\in [0,2\pi)$ and hence $F\in \dom(\mathcal{L})$.
So, to complete the proof of \eqref{3.36} it remains to note that $L$ is essentially self-adjoint on $\dom(L_{\min})$.

The following theorem describes the well-known fundamental properties of the spectrum of the operator $L$
(cf.\ \cite{KhoRof}, \cite{McG_65}, see also \cite{LanDah86,Zet05}).

\begin{theorem} \label{th_III.1}
 Assume that the operator $L$ is $J$-nonnegative. Then the spectrum of $L$  is real, purely continuous,
$\sigma(L)=\sigma_{\rm c}(L)$. Moreover, it consists of countably many closed intervals and
\begin{align}
\sigma(L) &=\big\{\lambda\in\CC\,\big|\, -1\leq \Delta(\lambda)
\leq 1\big\}=\bigcup_{t\in[0,2\pi)} \sigma(L(t)).  \label{3.12} 
\end{align}
\end{theorem}

\subsection{Regularity of critical points.}\label{ss_III.2}

The similarity problem for definitizable operators is closely connected to the regularity problem for critical points of the corresponding operator. For $J$-positive operators, only $0$ and $\infty$ may be  critical points. In this subsection we investigate the problem of regularity of critical points of $J$-positive operator $L$ defined by \eqref{I_02}.

\subsubsection{Regularity of $\infty$.}

We begin with the study of the critical point $\infty$. Note that $\infty$ is indeed a critical point of $L$ since the spectrum of $L$ accumulates at both $-\infty$ and $+\infty$.

\begin{theorem}\label{th_inftycond}
Let $q$ and $\omega$ be $2\pi$-periodic real functions and let the operator $L$ be defined by \eqref{I_02}. Assume that $\omega$ has a finite number of turning points on $[0,2\pi)$. 
Then:
\item $(i)$ \ \ $\infty$ is a regular critical point of $L$ if all turning points of $\omega$ are simple.
\item $(ii)$ \ \ if some turning point $x_0$ is not simple and, in addition, $\omega$ is odd and continuously differentiable in a punctured neighborhood of $x_0$ and the limit in \eqref{eq:fleige} exists, then $\infty$ is a singular critical point of $L$.
\end{theorem}
\begin{proof}
$(i)$ Assume that $\omega$ has only a finite number of turning points on $[0,2\pi)$. Without loss of generality we can assume that $\omega$ does not change sign at $x=0$.
Let $X=\{x_i\}_{i=1}^N$, $N<+\infty$, be the set of turning points of $\omega$ on $(0,2\pi)$. Set
\[
\mathcal{D}_0(t):=\big\{f\in \dom\big(\gt(t)\big):\ f(x_i)=0,\ x_i\in X\big\},
\]
where $\dom\big(\gt(t)\big)$ is the form domain of $A(t):=\ti{J}L(t)$ (see \eqref{II_1_03B}).

If all turning points are simple, then (for details see the proof of Theorem 3.6 in \cite{CurLan}) there exists a bounded and boundedly invertible $\ti{J}$-positive operator $W$ in $L^2([0,2\pi];|\omega|)$ such that $W\big(\dom\big(\gt(t)\big)\big)\subseteq \mathcal{D}_0(t)$.  Moreover, the operator $W$ does not depend on $t$ since $\omega$ does not change sign at $x=0$ and hence at $x=2\pi$.

According to the decomposition \eqref{3.31}, define the following operators in $\mathfrak{H}$
\[
\cW:=\frac{1}{2\pi}\int^{\oplus}_{[0,2\pi]} W\ dt,\qquad \mathcal{J}:=\frac{1}{2\pi}\int^{\oplus}_{[0,2\pi]} \widetilde{J}\ dt.
\]
Note that $\cG J\cG^{-1}=\mathcal{J}$. It is straightforward to check that $\cW$ is bounded and boundedly invertible $\mathcal{J}$-positive operator in $\gH$. Moreover, it is clear that $\cW:\mathcal{D}(\mathcal{A})\to\mathcal{D}(\mathcal{A})$, where $\mathcal{D}(\mathcal{A})$ denotes the form domain of $\mathcal{A}=\mathcal{JL}$. Therefore, by \cite[Theorem 3.5]{Cur_85}, $\infty$ is a regular critical point of $\mathcal{L}$. Hence, by \eqref{3.36} we conclude that $\infty$ is a regular critical point of $L$.

$(ii)$  If the weight $\omega$ has a turning point $x_0$ which is not simple and, moreover, $\omega$ is odd in a neighborhood of $x_0$, then, by Theorem \ref{th_II.1}$(ii)$ (cf. also \cite{Par03,Pyat05}), $\infty$ is a singular critical point for $L(t)$ for each $t\in [0,2\pi)$. Therefore, we can conclude that $\infty$ is also a singular critical point for $L$. Namely, to prove the last claim let us show that the regularity of the critical point $\infty$ of $L$ implies that $\infty$ is regular for $L(t)$ for almost all $t\in [0,2\pi)$.

Consider the spectral function $E_L(\Delta)$ of $L$ (for details we refer the reader to \cite{Lan82}). Note that $L$ has a spectral function since it is definitizable. Assume that $\infty$ is a regular critical point for $L$. Since $L$ commutes with the translation operator $T: f(t)\to f(t+2\pi)$, we conclude that its resolvent $(L-z)^{-1}$ and hence the spectral function $E_L(\Delta)$ also commute with $T$. Here $\Delta$ runs through all intervals with endpoints different from $\lambda = 0$. Therefore, (see, e.g., \cite[Theorem 5.13]{McG_65}, $E_L(\Delta)$ admits the representation
\[
\cG E_L(\Delta) \cG^{-1}=\frac{1}{2\pi}\int_{[0,2\pi]}^\oplus E_L(t;\Delta)dt,
\]
and, moreover, $\|E_{L}(\Delta)\|=\esssup_{t\in[0,2\pi)}\|E_L(t;\Delta)\|$.
Further, using the representation \eqref{3.36}, one clearly gets
\[
E_L(t;\Delta)=E_{L(t)}(\Delta),
\]
where $E_{L(t)}(\Delta)$ stands for the spectral function of the definitizable operator $L(t)$. Thus, we conclude that if $\infty$ is regular for $L$, then $\infty$ is regular for $L(t)$ for almost all $t\in [0,2\pi)$.
\end{proof}

\subsubsection{Regularity of $0$.}

We begin with several preliminary lemmas that provide more details on the spectrum of the $J$-positive operator $L$.
\begin{lemma}\label{lem_3.2}
Assume that $\omega\in L^1[0,2\pi]$ changes sign and the operator $L$ is $J$-positive.
Then $z=0$ does not belong to the Dirichlet spectrum $\sigma(L_D)$ and $s(2\pi,0)>0$ holds. Moreover, $\Delta(0)\ge 1$ and there are three possibilities:
\begin{description}
\item $(i)$ if $\Delta(0)\neq 1$, then $\Delta(0)>1$ and hence $0\in \rho(L)$,
\item $(ii)$ if $\Delta(0)= 1$ and $\Delta^{\bullet}(0)\neq 0$, then there is $\varepsilon>0$ such that \\
$\bullet$ \ either $[-\varepsilon,0]\subset\sigma(L)$ and $(0,\varepsilon)\subset\rho(L)$ if $\Delta^{\bullet}(0)> 0$\\
$\bullet$ \ or  $[0,\varepsilon]\subset\sigma(L)$ and $(-\varepsilon,0)\subset\rho(L)$ if $\Delta^{\bullet}(0)< 0$,
\item $(iii)$ if $\Delta(0)= 1$ and $\Delta^{\bullet}(0)= 0$, then $\Delta^{\bullet\bullet}(0)<0$  and there is $\varepsilon>0$ such that $[-\varepsilon,\varepsilon]\subset\sigma(L)$.
\end{description}
\end{lemma}
\begin{proof}
The operator $L$ is positive if and only if the operator $A=-d^2/dx^2+q(x)$ acting in $L^2(\R)$ is also positive. Moreover, the Flocke discriminants of $L$ and $A$ coincide at $z=0$ since $c(x,0)$ and $s(x,0)$, defined by \eqref{3.3}, solve the equation
 \[
 -y''(x)+q(x)y(x)=0,\quad x\in\R.
 \]
The latter yields (cf. \cite[Theorem 12.7]{Weid87}) $\Delta(0)\geq 1$, $s(2\pi,0)>0$, and hence $z=0$ does not belong to the Dirichlet spectrum $\sigma(L_D)$. 

Therefore, if $\Delta(0)\neq 1$, then $\Delta(0)>1$ and hence $0\notin\sigma(L)$.

$(ii)$ clearly follows from \eqref{3.12}, \eqref{3.7D} and the fact that the function $\Delta$ is entire.

In the case $\Delta(0)= 1$ and $\Delta^{\bullet}(0)= 0$, the proof of the inequality $\Delta^{\bullet\bullet}(0)<0$ is analogous to the proof of Proposition 3.3(b) in \cite{LanDah86} and we omit it. Further, combining the last inequality with \eqref{3.12} and \eqref{3.7D}, we prove the last claim.
\end{proof}
\begin{corollary}\label{cor:3.4}
If $\Delta(0)=1$, then the geometric multiplicity of $z=0$ as an eigenvalue of $L(0)$ equals $1$, $\dim\big(\ker L(0)\big)=1$. Moreover, its algebraic multiplicity is at most $2$ ($\ker L(0)^2=\ker L(0)^3$) and it equals $2$ ($\ker L(0)\neq \ker L(0)^2$) precisely when $\Delta^{\bullet}(0)=0$.
\end{corollary}
\begin{proof}
To prove the first claim let us  assume the converse, i.e., $\dim\big(\ker L(0)\big)=2$. Clearly, the latter is true if and only if all solutions of the equation $\ell[y]=0$ satisfy periodic boundary condition. Therefore, so is $s(x,0)$ and hence $s(2\pi,0)=s(0,0)=0$. However, the latter contradicts Lemma \ref{lem_3.2}.

Further, the operator $L(0)$ is $\ti{J}$-nonnegative since the operator $L$ is $J$-positive. Therefore, definitizability of $L(0)$ implies that the algebraic multiplicity of $z=0$ is at most $2$ (cf. Section 2 and \cite{Lan82}).

To prove the last claim observe that
\[
\Delta(z)-1=-\det\big(\cM(z)-I\big)/2.
\]
Noting that $s(2\pi,0)>0$ and using the standard arguments (see, e.g., \cite[\S I.3.7--8]{nai}), we see that  $z=0$ is a simple eigenvalue of $L(0)$ if and only if $z=0$ is a simple zero of $\Delta(z)-1=0$. Moreover, if $z=0$ is a stationary point of $\Delta(\cdot)$, then the resolvent $\big(L(0)-z\big)^{-1}$ has a pole of order $2$ at $z=0$. The latter also implies $\ker L(0)\neq \ker L(0)^2$ (see, e.g., \cite[\S I.3.9]{nai}).
\end{proof}
\begin{lemma}\label{lem_3.3}
\begin{equation}
\Delta^{\bullet}(z)=-s(\pi,z)\frac{1}{2} \int_{0}^{\pi}\psi_+(x,z)\psi_-(x,z)
\omega(x)dx , \quad z\in\CC.   \label{3.24}
\end{equation}
\end{lemma}
\begin{proof}
Combining the identity
\[
\int_0^{2\pi}\psi_+(x,z_1)\psi_-(x,z_2)\omega(x)dx=\frac{W(\psi_+(x,z_1),\psi_-(x,z_2))\mid_{x=0}^{2\pi}}{z_1-z_2},\quad z_1\neq z_2\in\CC,
\]
with \eqref{l24}, \eqref{3.21}, \eqref{3.6} and \eqref{3.20}, after straightforward calculation we arrive at \eqref{3.24}.
\end{proof}

Combining Lemmas \ref{lem_3.2} and \ref{lem_3.3}, we arrive at the following result.
\begin{proposition}\label{prop_3.1}
Let $L$ be $J$-positive and $0\in\sigma(L)$. Then $0$ is a critical point of $L$ if and only if
\begin{equation}\label{5.1}
\int_0^{2\pi}\psi_+(x,0)^2\omega(x)dx=0.
\end{equation}
\end{proposition}
\begin{proof}
Firstly, by Lemma \ref{lem_3.2}$(i)$--$(iii)$, observe that $0$ is a critical point of $L$ if and only if $\Delta(0)= 1$ and $\Delta^{\bullet}(0)=0$.
The latter implies $0\in \sigma(L(0))$. Moreover, by \eqref{3.20A}, the corresponding eigenfunction is $\psi_+(x,0)=\overline{\psi_+(x,0)}=\psi_-(x,0)$. Hence, by Lemma \ref{lem_3.3}, one gets
\[
\Delta^{\bullet}(0)=-s(2\pi,0)\int_0^{2\pi}\psi_+(x,0)^2\omega(x)dx.
\]
Finally, by Lemma \ref{lem_3.2}, $s(2\pi,0)>0$ and hence $\Delta^{\bullet}(0)=0$ if and only if \eqref{5.1} holds.
The proof is completed.
\end{proof}
\begin{corollary}\label{cor:3.7}
Let $q\equiv 0$. Then $0$ is a critical point of $L$ if and only if
\begin{equation}\label{5.2}
\int_0^{2\pi}\omega(x)dx=0.
\end{equation}
\end{corollary}
\begin{proof}
Since $q\equiv 0$, we get $c(x,0)\equiv 1$ and $s(x,0)=x$. Therefore, $\Delta(0)=1$ and hence $0\in \sigma(L)$. To complete the proof it remains to note that $\psi_+(x,0)\equiv 1$.
\end{proof}
\begin{remark}
%
Indefinite Sturm--Liouville operators on the half-line $\R_+$ have been studied in \cite{LanDah86} under the assumption $q\geq 0$. It was shown in \cite[Theorem 5.1]{LanDah86} that in this case $0$ is a critical point if and only if $q\equiv 0$ and \eqref{5.2} holds.
\end{remark}
Now we are ready to formulate the main result of this subsection.
\begin{theorem}\label{th_3.3}
Let $L$ be the $J$-positive operator defined by \eqref{I_02}.  Then the following are equivalent:
\begin{itemize}
\item $(i)$ \ $0$ is a critical point of $L$,
\item $(ii)$ \ $0$ is a singular critical point of $L$,
\item $(iii)$ \ $0\in \sigma(L(0))$ and \eqref{5.1} holds.
\end{itemize}
\end{theorem}
\begin{proof}
The equivalence $(i) \Leftrightarrow (iii)$ was proven in Proposition \ref{prop_3.1}.
Moreover, the implication $(ii)\Rightarrow (i)$ is clear. So, to complete the proof we only need to prove the implication $(iii) \Rightarrow (ii)$.

Assume that $\Delta(0)=1$ and \eqref{5.1} holds. The first equality yields the inclusion $0\in \sigma(L(0))$. Moreover, by Lemma \ref{lem_3.3}, $\Delta^{\bullet}(0)=0$. Hence, by Corollary \ref{cor:3.4}, $z=0$ is a nonsimple eigenvalue of the operator $L(0)$ defined by \eqref{II_1_01}, \eqref{II_1_03}, that is, $\ker L(0)=\Span\{\psi_+(x,0)\}$ and $\ker L(0)\neq \ker L(0)^2$. This means that $L(0)$ is not similar to a self-adjoint operator and $\|(L(0)-z)^{-1}\|=O(|z|^{-2})$ as $z\to 0$. Denote by $\phi(x)$ the solution of the equation $(L(0)\phi)(x)=\psi_+(x,0)$, which is orthogonal to $\psi_+(x,0)$ in $L^2([0,2\pi];|\omega|)$. Note that for $y\in\R_+$
\[
\|(L(0)-\I y)^{-1}\phi\|=\|\psi_+(x,0)/y^2+\I\phi(x)/y\|\ge \frac{C}{y^{2}},
\]
where $C=\|\psi_+(x,0)\|_{L^2([0,2\pi];|\omega|)}>0$.
Further, observe that the function $g(t,z):=\|(L(t)-z)^{-1}\phi\|$ is continuous in $t$ on $[0,2\pi)$ for all $z\in\CC_+$. Therefore,
\[
\esssup_{t\in[0,2\pi)}\|(L(t)-\I y)^{-1}\|\geq \frac{\max_{t\in[0,2\pi)}g(t,\I y)}{\|\phi\|}\ge \frac{C}{\|\phi\|\ y^{2}},\quad y\in\R_+. 
\]
Due to representation \eqref{3.36} and by \cite[Theorem VII.2.3]{BirSol_87},
\[
\|(L-\I y)^{-1}\|_{L^2(\R,\omega)}=\esssup_{t\in[0,2\pi)}\|(L(t)-\I y)^{-1}\|,\qquad y\in \R_+.
\]
Therefore, we get
\[
\|(L-\I y)^{-1}\|_{L^2(\R,\omega)}\geq  \frac{C}{\|\phi\|\ y^{2}},\quad y\in\R_+.
\]
The latter means that the operator $L$ has a spectral singularity at $z=0$ and hence $0$ is a singular critical point for $L$.
\end{proof}

Combining Theorem \ref{th_3.3}  with Corollary \ref{cor:3.7},  we obtain the following result.
\begin{corollary}\label{c5.3}
Let $q\equiv 0$. Then $0$ is a singular critical point of $L$ if and only if \eqref{5.2} holds.
\end{corollary}

\subsubsection{Similarity criterion.} \label{sss_III.2.1}

Combining Theorem \ref{th_3.3} with the regularity conditions of the critical point $\infty$, Theorem \ref{th_inftycond}, we arrive at the following necessary and sufficient similarity conditions for the operator $L$.
\begin{theorem}\label{th_3.4}
Assume that the weight $\omega$ has only a finite number of turning points on $[0,2\pi]$. Assume also that the operator $L$ defined by \eqref{I_02} is $J$-positive. Then  $L$ is similar to a self-adjoint operator if the following conditions hold:
\begin{itemize}
\item $(i)$ either $0\notin \sigma(L(0))$ or $\int_0^{2\pi}\psi_+(x,0)^2\omega(x)dx\neq 0$,
\item $(ii)$ all turning points of $\omega$ are simple.
\end{itemize}

If, in addition, $\omega$ is odd and continuously differentiable in a punctured neighborhood of each turning point, and limit \eqref{eq:fleige} exists for every turning point, then conditions $(i)$-$(ii)$ are also necessary.
\end{theorem}
\begin{proof}
Since the operator $L$ is $J$-positive, it is similar to a self-adjoint operator if and only if all its critical points are regular. Theorems \ref{th_inftycond} and \ref{th_3.3} complete the proof.
\end{proof}

\subsection{Examples.}\label{ss_exmpl}

\begin{example}\label{ex_I}
Let $\omega(x)=\sgn \sin x$, i.e.,
\begin{equation}\label{ex_01}
L=(\sgn \sin x)\frac{d^2}{dx^2},\qquad \dom(L)=W^{2,2}(\R).
\end{equation}
First, observe that
\[
\int_{0}^{2\pi}\omega(x)dx=\int_0^{2\pi}\sgn (\sin x) dx=0,
\]
and hence, by Corollary \ref{c5.3}, $0$ is a singular critical point of $L$. Therefore (see also Theorem \ref{th_3.4}), the operator $L$ is not similar to a self-adjoint operator. However, $\infty$ is a regular critical point for $L$ since the weight function has only 2 turning points on $[0,2\pi)$ and both these points are simple (clearly, $\omega$ has the form \eqref{e BC Simple_sec} in a neighbourhood of $x=\pi n$, $n\in\Z$).
\end{example}
Let us note that the operator \eqref{ex_01} gives the example of a \emph{$J$-positive} Sturm--Liouville operator with the singular critical point $0$. Examples of \emph{$J$-nonnegative} Sturm--Liouville operators with the singular critical point $0$ was first constructed in \cite{KarKos08, KarKos09} (see also \cite[Section 5]{KAM_09} and \cite[Section 5]{Kar_09}), however, in these examples the operators have a nontrivial kernel.

\begin{example}\label{ex_II}
Let $a\in (0,\pi)$ and let $\omega_a$ be a $2\pi$-periodic function defined on $\R$ by
\begin{equation}\label{ex_02}
\omega_a(x)=\sgn (x-a),\qquad x\in (-\pi,\pi]. 
\end{equation}
First, observe that $\infty$ is a regular critical point for $L$ since all turning points of $\omega$ are simple.
However,
\[
\int_{0}^{2\pi}\omega_a(x)dx=\int_0^{2\pi}\sgn (x-a) dx=-2a,
\]
and hence, by Corollary \ref{c5.3}, $0$ is a singular critical point of $L_a$ if and only if $a=0$. Moreover, by Theorem \ref{th_3.4}, the operator $L$ is not similar to a self-adjoint operator if and only if $a=0$.
\end{example}

\ack{The author is grateful to Illya Karabash for a careful reading of the manuscript and numerous helpful remarks. The author thanks Daphne Gilbert and Mark Malamud for useful discussions. The author thanks the referees for useful critical comments, which have helped to improve the exposition.

The support from the IRCSET Postdoctoral
Fellowship Program is gratefully acknowledged.}


\quad
\\
Aleksey Kostenko, \\
\quad \\
Institute of Applied Mathematics and Mechanics, NAS of Ukraine,\\
R. Luxemburg str., 74, Donetsk 83114, UKRAINE\\
\quad \\
and\\
\quad\\
School of Mathematical Sciences,\\
Dublin Institute of Technology,\\
Kevin Street, Dublin 8,\\
IRELAND\\
\emph{e-mail:} duzer80$@$gmail.com

\end{document}